\documentclass{amsart}

\usepackage[english]{babel}
\usepackage[latin1]{inputenc}
\usepackage{amsmath}
\usepackage{amsthm}
\usepackage{amssymb}
\usepackage{tikz}

\usepackage{imakeidx}
\usepackage{appendix}
\usepackage{tikz-cd}
\usepackage{verbatim}
\usepackage{enumitem}
\usepackage{multicol}

\usepackage{hyperref}
\usepackage{cleveref}

\usepackage{mathtools}

\newtheorem{thm}{Theorem}[section] 

\newtheorem{prop}[thm]{Proposition}

\newtheorem{lemma}[thm]{Lemma}

\theoremstyle{definition}

\newtheorem{example}[thm]{Example}
\theoremstyle{remark}

\newtheorem{rmk}[thm]{Remark}
\numberwithin{equation}{section}

\newcommand{\Q}{\mathbb Q}
\newcommand{\Z}{\mathbb Z}
\newcommand{\G}{\mathbb G}

\newcommand{\Spec}{\operatorname{Spec}}
\renewcommand{\c}{\subseteq}
\newcommand{\A}{\mathbb A}
\newcommand{\mc}[1]{\mathcal{#1}}
\newcommand{\cl}{\overline}
\newcommand{\set}[1]{\{#1\}}
\renewcommand{\phi}{\varphi}
\newcommand{\on}[1]{\operatorname{#1}}
\newcommand{\ang}[1]{\left \langle{#1}\right \rangle}

\DeclareFontFamily{U}{wncy}{}
\DeclareFontShape{U}{wncy}{m}{n}{<->wncyr10}{}
\DeclareSymbolFont{mcy}{U}{wncy}{m}{n}
\DeclareMathSymbol{\Sh}{\mathord}{mcy}{"58}

\title{On the Noether Problem for torsion subgroups of tori}
\author{Federico Scavia}

\begin{document}
	\begin{abstract}
		We consider the Noether Problem for stable and retract rationality for the sequence of $d$-torsion subgroups $T[d]$ of a torus $T$, $d\geq 1$. We show that the answer to these questions only depends on $d\pmod{e(T)}$, where $e(T)$ is the period of the generic $T$-torsor. When $T$ is the norm one torus associated to a finite Galois extension, we find all $d$ such that the Noether Problem for retract rationality has a positive solution for $d$. We also give an application to the Grothendieck ring of stacks.
	\end{abstract}	
	
	\maketitle

	\section{Introduction}
	Let $k$ be a field, $G$ a finite group (i.e. a finite constant group scheme over $k$) and $V$ a faithful $k$-linear $G$-representation. The Noether Problem asks whether the quotient variety $V/G$ is rational, that is, birational to some affine space over $k$. This question originated in Noether's work on the Inverse Galois Problem \cite{noether1917gleichungen}. If $k = \Q$, and $V/G$ is rational over $G$ for some $V$, then Hilbert's Irreducibility Theorem implies that $G$ arises as a Galois group over $\Q$; see \cite[\S 3.3, \S 5.1]{jensen2002generic}.
	
	It turns out that the Noether Problem does not always have a positive solution. Swan and Voskresenskii gave counter-examples to the Noether Problem for some cyclic groups of prime order ($\Z/47\Z$, $\Z/113\Z$, $\Z/223\Z$,...) in \cite{swan1969invariant} and \cite{voskresenskii1970question}. Later, Lenstra showed in \cite{lenstra1974rational} that the smallest group for which the Noether Problem fails is the cyclic group $\Z/8\Z$, and further he gave a complete classification of abelian groups for which the Noether Problem fails. 
	
	% In the effort of finding negative answers to the Noether Problem, several variants have been formulated. 
	Subsequent work naturally led to several variants of the original Noether Problem. For example, one may ask if 
	the variety $V/G$ is stably rational, or retract rational; see \cite[\S 1]{colliot2007rationality} for the definitions.
	The examples of Swan and Voskresenskii give a negative answer to the Noether Problem for stable rationality and, by a result of Saltman \cite[Theorem 4.12]{saltman1984retract}, a positive answer for retract rationality. On the other hand $\Z/8\Z$ answers both problems negatively.
	
	One may also consider the Noether Problem for more general group schemes $G$. Let $G$ be a linear algebraic group (not necessarily smooth), and let $V$ be a finite-dimensional generically free representation of $G$. There exists a dense $G$-invariant open subset of $V$ such that a geometric quotient $U/G$ exists and $U\to U/G$ is a $G$-torsor. The variety $U/G$ may be regarded as an approximation of the classifying stack $BG$. For example, by the no-name lemma \cite[Lemma 2.1]{reichstein2006birational}, the stable (retract) rationality of $U/G$ does not depend on the representation $V$ but only on $G$. We say that $BG$ is stably (retract) rational if so is $U/G$, that is, if the Noether Problem for stable (retract) rationality has an affirmative answer; see also \cite[\S 3]{florence2018genus0}.
	
	Let $k$ be a field and $T$ a $k$-torus. For every $d\geq 1$, we will consider the Noether Problem for stable rationality and retract rationality for the torsion subgroup $T[d]$. Note that these questions are not covered by the above-mentioned results of Lenstra~\cite{lenstra1974rational}. Lenstra solved the Noether Problem for
	finite abelian groups. The torsion subgroup schemes $T[d]$ we will consider are finite and abelian, but not necessarily constant.
	
	In \Cref{periodicity} we show that the answer to each of these versions of the Noether Problem for the torsion subgroups $T[d]$ 
	(of a particular torus $T$) is periodic in $d$:
	
	\begin{thm}\label{periodicity}
		Let $T$ be a $k$-torus and $m,n\geq 1$.
		\begin{enumerate}[label=(\alph*)]
			\item\label{periodicity1} If $m\equiv \pm n\pmod{e(T)}$, then $BT[m]$ and $BT[n]$ are stably birational. 
			\item\label{periodicity2} If $e(T)\mid n$, then $BT[n]$ is stably birational to $T\times BT$.
			\item\label{periodicity5} If $n\equiv \pm 1\pmod{e(T)}$, then $BT[n]$ is stably rational.
			\item\label{periodicity4} If $\ang{m}=\ang{n}$ in $\Z/e(T)\Z$, then $BT[m]$ is retract rational if and only if $BT[n]$ is. In particular, $BT[m]$ is retract rational when $m$ is invertible in $\Z/e(T)\Z$.
		\end{enumerate}
	\end{thm}
	Here $e(T)$ denotes the period of a generic $T$-torsor (see \Cref{preliminaries} for the definition). If $G$ is a splitting group for $T$, we have $e(T)\mid |G|$; see \Cref{period}\ref{period1}.
	
	To illustrate \Cref{periodicity}, we investigate more closely the Noether Problem for retract rationality of $T[d]$ in the case where $T=R^{(1)}_{L/k}(\G_m)$ is a norm-one torus and $L/k$ is a finite Galois extension. By \Cref{reducetosylow} it is enough to consider the case when $G$ is a $p$-group. For every $L/k$ and every $d\geq 1$, we determine whether $BT[d]$ is retract rational or not.
	
	\begin{thm}\label{classify}
		Let $L/k$ be a finite Galois extension such that $G:=\on{Gal}(L/k)$ has order $p^n$ for some prime number $p$, and let $T:=R^{(1)}_{L/k}(\G_m)$. Assume that $p$ is odd.
		\begin{enumerate}[label=(\alph*)]
			\item\label{classify1} If $G$ is cyclic, then $BT[d]$ is retract rational for every $d$.
			\item\label{classify2} If $G$ is not cyclic, then $BT[d]$ is retract rational if and only if $p\nmid d$.
		\end{enumerate}
		Assume that $p=2$.
		\begin{enumerate}[label=(\alph*')]
			\item\label{classify3} If $G$ is cyclic, then $BT[d]$ is retract rational for every $d$.
			\item\label{classify4} If $G$ is a dihedral group, then $BT[d]$ is retract rational if and only if $4\nmid d$.
			\item\label{classify5} Otherwise, $BT[d]$ is retract rational if and only if $d$ is odd.
		\end{enumerate}
	\end{thm}
	Assume that $\on{char}k=p$ is positive. Then it is well-known that for every finite $p$-group $\Gamma$, $B\Gamma$ is stably rational; see \cite[\S 5.6]{jensen2002generic}. By contrast, \Cref{classify} gives examples of finite group schemes $A$ of order a power of $p$ for which $BA$ is not even retract rational. Note that such $A$ are non-reduced.
	
	In \Cref{grotring} we give an application of \Cref{periodicity} to the Grothendieck ring of $k$-stacks $K_0(\on{Stacks}_k)$.
	
	%Our main references are \cite[Chapter 2]{voskresenskii2011algebraic}, \cite{colliot1977r} and \cite[Chapter 2]{lorenz2006multiplicative}.
	
	\section{Preliminaries}\label{preliminaries}
	Let $G$ be a linear algebraic group over $k$. If $V$ is a generically free representation of $G$, there exists a non-empty open $G$-invariant subset $U\c V$ together with a $G$-torsor $U\to Z$, where $Z$ is a $k$-variety. 
	%We may always assume that $\on{dim}Z>0$, so that $k(Z)$ is infinite. 
	Such a torsor $U\to Z$ is versal.
	We say that $BG$ is stably birational to a variety $X$ if $Z$ is birational to $X$. We say that $BG$ is stably rational or retract rational if $Z$ is. Different choices of $V$ and $U$ yield stably birational $Z$, hence the definitions are independent of the choice of $V$ and $U$; see \cite[\S 5]{merkurjev2018versal}.
	
	Let $T$ be a torus over $k$. We will denote by $\hat{T}$ the character lattice of $T$. Recall that this is $\on{Hom}_{k_s}(T_{k_s},\G_{m,k_s})$, where $k_s$ denotes a separable closure of $k$. The association $T\mapsto \hat{T}$ establishes an anti-equivalence between the categories of $k$-tori and the category of $\Z$-free continuous $\on{Gal}(k)$-modules of finite rank; see \cite[\S 3.4]{voskresenskii2011algebraic}. We will write $T'$ for the dual torus of $T$, that is, the $k$-torus whose character is dual to $\hat{T}$.
	
	Consider a short exact sequence \[1\to T\to S\xrightarrow{\phi} Q\to 1\] where $S$ is quasi-split. Then $S$ is an open subset of an affine space $V$, and the multiplication action of $T$ on $S$ extends to $V$. It follows that $\phi$ is a standard $T$-torsor, in the sense of \cite{merkurjev2018versal}. In particular, $BT$ is stably birational to $Q$.
	
	The generic fiber of $\phi$ is a $T$-torsor over $k(S)$, called the \emph{generic} $T$-torsor. If $E\to \Spec K$ is a $T$-torsor, its \emph{period} is its order in the group $H^1(K,T)$. We denote by $e(T)$ the period of a generic $T$-torsor. By \cite[Proposition 1.1]{merkur2010periods}, $e(T)$ is divisible by the period of every $T$-torsor. In particular, it does not depend on the choice of the resolution (\ref{resolution}). 
	
	\begin{lemma}\label{period}
		Let $T$ be a $k$-torus, with splitting group $G$. Then:
		\begin{enumerate}[label=(\alph*)]
			\item\label{period1} $e(T)\mid |G|$;
			\item\label{period2} if $T'$ denotes the dual torus of $T$, then $e(T)=e(T')$.
		\end{enumerate}
	\end{lemma}
	
	\begin{proof}
		\ref{period1}. Let $l\supseteq k$ be the splitting field of $T$. The extension $l/k$ is Galois, with Galois group $G$. Let $E\to\Spec K$ be a generic $T$-torsor, where $K$ is a field containing $k$, and denote by $L$ the compositum of $K$ and $l$ inside some fixed algebraic closure of $K$ containing $l$. Then $L$ splits $T_K$ and so $H^1(L,T_L)=H^1(L,\G_m^{\on{rk}T})=0$. The restriction-corestriction sequence \[H^1(K,T)\to H^1(L,T_L)\to H^1(K,T)\] shows that $H^1(K,T)$ is $[L:K]$-torsion, hence $e(T)\mid [L:K]$. By basic Galois theory $[L:K]\mid [l:k]=|G|$, so $e(T)\mid |G|$, as desired.
		
		\ref{period2}. Set $M:=\hat{T}$ and fix a coflasque resolution $0\to N\to U \to M\to 0$, corresponding to a class $\alpha\in \on{Ext}^1_G(M,N)$. Let $M'$, $N'$ be the dual lattices of $M$ and $N$, respectively. There is a $G$-equivariant isomorphism \[\on{Hom}_{\Z}(M,N)\to \on{Hom}_{\Z}(N',M')\] given by taking the transpose. By \cite[\S III, Proposition 2.2]{brown1994cohomology} we obtain a commutative diagram of group isomorphisms
		\begin{equation}\label{dualorder}\begin{tikzcd}
		H^1(G,\on{Hom}_{\Z}(M,N)) \arrow[r,"\sim"] \arrow[d,"\wr"] & \on{Ext}^1_G(M,N) \arrow[d,"\eta"]\\
		H^1(G,\on{Hom}_{\Z}(N',M')) \arrow[r,"\sim"] & \on{Ext}^1_G(N',M'),
		\end{tikzcd}\end{equation}
		where $\eta$ is defined by sending an extension $0\to N\to P\to M\to 0$ to its dual $0\to M'\to P'\to N'\to 0$.  By definition, $e(T)$ is the order of $\alpha$ in $\on{Ext}^1_G(M,N)$. By \cite[Theorem 3.2]{merkur2010periods}, $e(T')$ is equal to the order of the class $\alpha'$ of the dual sequence $0\to M'\to U'\to N'\to 0$ in $\on{Ext}^1_G(N',M')$. Now (\ref{dualorder}) shows that $\alpha$ and $\alpha'$ have the same order, hence $e(T)=e(T')$. 
	\end{proof}
	
	\section{Proof of Theorem \ref{periodicity}}\label{noethertorsion} 
	
	Let $T$ be a torus over $k$. Fix a coflasque resolution of $T$ \begin{equation}\label{resolution}1\to T\to S\xrightarrow{\phi} Q\to 1\end{equation} and the corresponding lattice sequence \begin{equation}\label{resolution'}0\to \hat{Q}\to \hat{S}\to\hat{T}\to 0.\end{equation} Consider the following commutative diagram with exact rows and columns:
	\begin{equation}\label{noethertorsion1}	
	\begin{tikzcd}
	&&& 1\arrow[d]\\
	& 1 \arrow[d] && T \arrow[d]\\ 
	1 \arrow[r] & T[d] \arrow[r] \arrow[d] & S \arrow[r,"\psi"] \arrow[equal]{d} & T_d \arrow[r] \arrow[d] & 1\\
	1 \arrow[r] & T \arrow[r] \arrow[d,"d"] & S \arrow[r,"\phi"] & Q \arrow[r] \arrow[d] & 1\\
	& T \arrow[d] && 1\\
	& 1
	\end{tikzcd}
	\end{equation}	
	Note that the second row is (\ref{resolution}). The torus $T_d$ is defined so as to make the first row exact, and the copy of $T$ in the upper right corner is identified with the copy of $T$ on the lower left corner via the connecting homomorphism given by the snake lemma.
	
	The main ingredient for the proof of \Cref{periodicity} is the following observation.
	
	\begin{prop}\label{multiple}
		Let $\alpha\in \on{Ext}^1_G(\hat{T},\hat{Q})$ be the class of (\ref{resolution'}), and $\alpha_d$ the class of the sequence \[0\to \hat{Q}\to\hat{T}_d\to\hat{T}\to 0\] in $\on{Ext}^1_G(\hat{T},\hat{Q})$, for every $d\geq 1$. Then $\alpha_d=d\cdot\alpha$. 
		
		In particular, $T_m\cong T_n$ when $m\equiv \pm n\pmod{e(T)}$, and if $e(T)\mid d$, then $T_d\cong T\times Q$.
	\end{prop}
	
	\begin{proof}
		Consider the diagram (\ref{noethertorsion1}). Since $S$ is quasi-split, $S\to T_d$ is a generic $T[d]$-torsor, so $BT[d]$ is stably birational to $T_d$. From the last column and the second row of (\ref{noethertorsion1}) we obtain the following diagram with exact rows:
		\begin{equation}\label{noethertorsion2}
		\begin{tikzcd}
		1 \arrow[r] & T \arrow[r] \arrow[d,"d"] & S \arrow[r,"\phi"] \arrow[d,"\psi"] & Q \arrow[r] \arrow[equal]{d} & 1\\
		1 \arrow[r] & T \arrow[r] & T_d \arrow[r] & Q \arrow[r] & 1.
		\end{tikzcd}
		\end{equation}
		The commutative square on the right comes directly from (\ref{noethertorsion1}). That the map $T\to T$ is given by $t\mapsto t^d$ follows from our previous identification of the two copies of $T$ via the connecting map of the snake lemma: if $t\in T(\cl{k})$, then, viewing $T(\cl{k})\c T_d(\cl{k})$, there exists $s\in S(\cl{k})$ such that $\psi(s)=t$, and since $\phi(s)=1$ we see that $s\in T(\cl{k})$. The connecting homomorphism is defined by sending $s\mapsto s^d$, but thanks to our identification it is the identity, so $s^d=t$. In terms of lattices, (\ref{noethertorsion2}) yields
		\begin{equation}\label{noethertorsion3}
		\begin{tikzcd}
		0 \arrow[r] & \hat{Q} \arrow[r] \arrow[equal]{d} & \hat{T}_d \arrow[r] \arrow[d] & \hat{T} \arrow[r] \arrow[d,"d"] & 0\\
		0 \arrow[r] & \hat{Q} \arrow[r] & \hat{S} \arrow[r] & \hat{T} \arrow[r] & 0.
		\end{tikzcd}
		\end{equation}
		By \cite[Theorem 3.1]{merkur2010periods}, $e(T)$ equals the order of the class $\alpha\in\on{Ext}^1_G(\hat{T},\hat{Q})$ of the sequence \begin{equation}
		\label{reductionlattice}0\to \hat{Q}\to \hat{S}\to \hat{T}\to 0.\end{equation} 
		By definition, $\alpha_d\in\on{Ext}^1_G(\hat{T},\hat{Q})$ is the class of \begin{equation}\label{classn}0\to \hat{Q}\to \hat{T}_d\to \hat{T}\to 0.\end{equation}
		The long exact sequence for the functor $\on{Hom}_G(\hat{T},-)$ associated to (\ref{noethertorsion3}) reads 
		\[
		\begin{tikzcd}
		0 \arrow[r] & \on{Hom}_G(\hat{T},\hat{Q}) \arrow[r] \arrow[equal]{d} & \on{Hom}_G(\hat{T},\hat{T}_d) \arrow[r] \arrow[d] & \on{Hom}_G(\hat{T},\hat{T}) \arrow[r,"\partial_1"] \arrow[d,"d"] & \on{Ext}^1_G(\hat{T},\hat{Q}) \arrow[equal]{d}\\
		0 \arrow[r] & \on{Hom}_G(\hat{T},\hat{Q}) \arrow[r] & \on{Hom}_G(\hat{T},\hat{S}) \arrow[r] & \on{Hom}_G(\hat{T},\hat{T}) \arrow[r,"\partial_2"] & \on{Ext}^1_G(\hat{T},\hat{Q})
		\end{tikzcd}
		\]
		The classes $\alpha$ and $\alpha_d$ are the images of $\on{id}_{\hat{T}}$ under the boundary maps $\partial_1$ and $\partial_2$, respectively. Since $\partial_1(\on{id}_{\hat{T}})=d\cdot \partial_1(\on{id}_{\hat{T}})$, we deduce that $\alpha_d=d\cdot\alpha$. 
		
		Since $\alpha$ has order $e(T)$, if $m\equiv n\pmod{e(T)}$ then $T_m\cong T_n$. Recall that if $M,N$ are $G$-lattices and $\gamma\in \on{Ext}_G^{1}(M,N)$ is the class of \[0\to N\to P\xrightarrow{\eta} M\to 0,\] then $-\gamma$ is the class of \[0\to N\to P\xrightarrow{-\eta} M\to 0.\] Hence $T_m\cong T_n$ when $m\equiv \pm n\pmod{e(T)}$.
		Finally, if $e(T)\mid d$, then $\alpha_d=0$, so $\hat{T}_d\cong \hat{T}\oplus \hat{Q}$.
	\end{proof}
	
	Before proving \Cref{periodicity}, we need the following lemma.
	
	\begin{lemma}\label{surjretract}
		Let $T_1$ and $T_2$ be two $k$-tori split by a finite Galois extension with Galois group $G$. Let $f: T_1\to T_2$ be a surjective homomorphism of tori of finite degree prime to $|G|$. Then $T_1$ is retract rational if and only if $T_2$ is retract rational.
	\end{lemma}
	
	\begin{proof}
		Let $1\to F_i\to S_i\to T_i\to 1$ be a flasque resolution of ${T}_i$ for $i=1,2$. By \cite[Theorem 3.14(a)]{saltman1984retract}, $T_i$ is retract rational if and only if $\hat{F}_i$ is invertible, so it suffices to show that $\hat{F}_1$ is invertible if and only if $\hat{F}_2$ is.
		
		Set $a:=\on{deg}f$. If $p$ is a prime that does not divide $a$, the map $\phi:T_1\to T_2$ induces an isomorphism $(\hat{T}_1)_{(p)}\cong (\hat{T}_2)_{(p)}$ of $\Z_{(p)}[G]$-modules. Construct the pushout diagram 
		\begin{equation}
		\begin{tikzcd}
		& 0 & 0\\
		& (\hat{F}_2)_{(p)} \arrow[u] \arrow[r] & (\hat{F}_2)_{(p)} \arrow[u]\\ 
		0 \arrow[r] & (\hat{S}_2)_{(p)} \arrow[u] \arrow[r] & X \arrow[r] \arrow[u] & (\hat{F}_1)_{(p)} \arrow[r] & 0\\
		0 \arrow[r] & (\hat{T}_1)_{(p)} \arrow[u] \arrow[r] & (\hat{S}_1)_{(p)} \arrow[r] \arrow[u] & (\hat{F}_1)_{(p)} \arrow[r] \arrow[u] & 0 \\
		& 0 \arrow[u] & 0 \arrow[u] \\
		\end{tikzcd}
		\end{equation}	
		Since the $\hat{F}_i$ are flasque and the $\hat{S}_i$ are permutation for $i=1,2$, by \cite[Lemme 1(vii)]{colliot1977r} \[\on{Ext}^1_{\Z_{(p)}[G]}((\hat{F}_i)_{(p)},(\hat{S}_{3-i})_{(p)})=\on{Ext}^1_G(\hat{F}_i, \hat{S}_{3-i})_{(p)}=0.\] It follows from the diagram that $(F_1\oplus S_2)_{(p)}\cong X \cong (F_2\oplus S_1)_{(p)}$ for every prime not dividing $a$. Since $|G|$ and $a$ are coprime, this holds for every $p\mid |G|$. By \cite[31.3(ii)]{curtis1966representation}, it follows that $F_1$ is invertible if and only if $F_2$ is (for us $\Lambda=\Z[G]$ and so $S(\Lambda)$ is the set of prime divisors of $|G|$). 
	\end{proof}

	\begin{proof}[Proof of \Cref{periodicity}]
		\ref{periodicity1}.	By construction $BT[d]$ is stably birational to $T_d$ for every $d\geq 1$, so the conclusion follows from \Cref{multiple}.
		
		\ref{periodicity2}. If $e(T)\mid n$, by \Cref{multiple} the sequence \[1\to T\to T_n\to Q\to 1\] splits, so $T_n\cong T\times Q$. Since $BT$ is stably birational to $Q$ and $BT[n]$ is stably birational to $T_n$, we conclude that $BT[n]$ is stably birational to $T\times BT$. 
		
		\ref{periodicity5}. Since $T[1]$ is trivial, $T_1=S$ is rational. Now \ref{periodicity5} follows from \ref{periodicity1}.
		
		\ref{periodicity4}. Let $m,n\geq 1$ such that $\ang{m}=\ang{n}$ in $\Z/e(T)\Z$. There exists $a\geq 1$ invertible modulo $e(T)$ such that $m\equiv an \pmod{e(T)}$. By \Cref{multiple}, ${T}_{m}\cong {T}_{an}$. By \Cref{surjretract}, $T_{m}$ is retract rational if and only if $T_{am}$ is. The second statement is now a consequence of \ref{periodicity2} and \ref{periodicity5}.
	\end{proof}
	
	\begin{rmk}
		Using \Cref{period}\ref{period1}, we see that \Cref{periodicity} remains valid if we substitute $e(T)$ with $|G|$.
	\end{rmk}
	
	\section{Norm one tori}
	Let $T:=R_{L/k}^{(1)}(\G_m)$ and $L/k$ be a finite Galois extension with Galois group $G$. By \cite[Example 4.1]{merkur2010periods}, $e(T')=[L:k]=|G|$. Using \Cref{period}\ref{period2}, we deduce $e(T)=|G|$. 
	
	Define $T_d$ via the diagram (\ref{noethertorsion1}) where $S=R_{L/k}(\G_m)$, and $T$ is embedded in $S$ as the kernel of the norm map, so that $Q=\G_m$. By construction we have a short exact sequence \[1\to T[d]\to R_{L/k}(\G_m)\to T_d\to 1.\] For every subgroup $G'$ of $G$, let $\sigma_{G'}:=\sum_{g\in G'}e_g\in \Z[G]$, where $\set{e_g}$ is the standard basis of $\Z[G]$. We have $\hat{T}=\Z[G]/\ang{\sigma_G}$, so the sequence \[1\to T[d]\to T\xrightarrow{d} T\to 1\] shows that the character module of $T[d]$ is $(\Z/d\Z[G])/\ang{\sigma}$. It follows that \[\hat{T}_d=\ang{d\Z[G],\sigma_G}=\ang{dg, \sigma_G}_{g\in G}.\] We may thus construct the exact sequence 
	\begin{equation}\label{norm1torsion0}0\to \Z\to \Z[G]\oplus \Z\to \hat{T}_d\to 0\end{equation} where the first map sends $1$ to $(1,\dots,1;-d)$ and the second map sends $(g,0)\mapsto dg$ and $(0,1)\mapsto \sigma_G$. 
	
	If $G$ is a finite group, $M$ a $G$-module and $i\in \Z$, we write $H^i(G,M)$ for the $i$-th group of Tate cohomology, group denoted by $\hat{H}^0(G,M)$ if $i$ is explicitly set equal to $0$.
	
	\begin{lemma}\label{finalisom}
		For every $i\in\Z$, there is a natural isomorphism \[{H}^i(G,\hat{T}_d)\cong {H}^i(G,\Z/d\Z).\]
	\end{lemma}
	
	\begin{proof}
		Let $U_d$ be the character module of $T[d]$. The restriction of the norm exact sequence defining $R^{(1)}_{L/k}(\G_m)$ \[1\to T[d]\to R_{L/k}(\mu_d)\to \mu_d\to 1\] yields the exact sequence of $G$-modules \[0\to\Z/d\Z\to \Z/d\Z[G]\to U_d\to 0.\] We also have the sequence \[0\to\hat{T}_d\to\Z[G]\to U_d\to 0,\] corresponding to the first row of (\ref{noethertorsion1}). By Shapiro's lemma $H^i(G,\Z[G])=0$ and $H^{i}(G,\Z/d\Z[G])=0$. Looking at the associated cohomology long exact sequences, we deduce \[H^i(G,\hat{T}_d)\cong H^{i-1}(G,U_d)\cong H^i(G,\Z/d\Z).\qedhere\]
	\end{proof}
	
	Recall that if $M$ is a $G$-module and $i\in\Z$, then \[\Sh^i(G,M):=\on{ker}(H^i(G,M)\to \oplus_{g\in G}H^i(\ang{g},M)).\] If $M$ is a $G$-lattice, by \cite[Proposition 2.9.2(a)]{lorenz2006multiplicative} \begin{equation}\label{shah1}\Sh^2(G,M)\cong H^1(G,[M]^{fl}).\end{equation}
	\begin{lemma}\label{sha}
		Let $G$ act trivially on $\Z/d\Z$, and assume that $\Sh^2(G,\Z/d\Z)\neq 0$. Then $BT[d]$ is not retract rational.
	\end{lemma}
	
	\begin{proof}
		By \Cref{finalisom} we have a commutative square with horizontal isomorphisms
		\begin{equation}\label{shadiag}
		\begin{tikzcd}
		{H}^2(G,\hat{T}_d) \arrow[r,"\sim"] \arrow[d] & H^2(G,\Z/d\Z) \arrow[d]\\
		\oplus_{g\in G}{H}^2(\ang{g},\hat{T}_d) \arrow[r,"\sim"] & \oplus_{g\in G}H^2(\ang{g},\Z/d\Z)
		\end{tikzcd}
		\end{equation}
		The diagram shows that $\Sh^2(G,\hat{T}_d)\cong \Sh^2(G,\Z/d\Z)$, thus $\Sh^2(G,\hat{T}_d)\neq 0$. By (\ref{shah1}) $H^1(G,[\hat{T}_d]^{fl})\neq 0$, hence $[\hat{T}_d]^{fl}$ is not invertible. It follows that $T_d$ is not retract rational, so $BT[d]$ is not retract rational.
	\end{proof}

	\begin{lemma}\label{subgroup}
		Let $G$ be a finite group, $H$ a subgroup of $G$ and $d\geq 1$ an integer. Let $L/k$ be a Galois extension with Galois group $G$, and let $k':=L^H$. Then $BR^{(1)}_{L/k}(\G_m)[d]\times_kk'$ is stably birational to $BR^{(1)}_{L/k'}(\G_m)[d]$ over $k'$.
		
		In particular, if $R^{(1)}_{L/k'}(\G_m)[d]$ is not retract rational over $k'$, then $R^{(1)}_{L/k}(\G_m)[d]$ is not retract rational over $k$.
	\end{lemma}
	
	\begin{proof}
		For every finite group $\Gamma$, denote by $\set{e_{\gamma}}$ the standard basis for $\Z[\Gamma]$, let $\sigma_{\Gamma,d}=(\sum_{\gamma\in \Gamma}e_{\gamma},-d)\in \Z[\Gamma]\oplus\Z$, and set $M_{\Gamma,d}:=(\Z[\Gamma]\oplus \Z)/\ang{\sigma_{\Gamma,d}}$. If $K/k$ is a Galois extension with Galois group $\Gamma$ and $T:=R^{(1)}_{K/k}(\G_m)$, by (\ref{norm1torsion0}) $\hat{T}_d=M_{\Gamma,d}$. Therefore, to prove the lemma it is enough to exhibit an isomorphism of $H$-lattices \[M_{G,d}\cong \Z[H]^{r-1}\oplus M_{H,d},\]  where $r:=[G:H]$. The choice of a set of representatives of $G/H$ gives an isomorphism of $H$-lattices $\Z[G]\cong \Z[H]^r$, hence an isomorphism $\Z[G]\oplus\Z\cong \Z[H]^r\oplus\Z$ which sends $\sigma_G$ to $\tau:=(1,\dots,1,-d)$. It follows that $M_{G,d}\cong (\Z[H]^r\oplus\Z)/\ang{\tau}$. For every $i=1,\dots, r$, denote by $v_i:=(0,\dots,0,e_1,0,\dots,0)\in \Z[H]^r$, where $e_1\in \Z[H]$ is the basis vector corresponding to the identity $1\in H$, and $e_1$ appears in the $i$-th entry of $v_i$. Consider the automorphism \[f:\Z[H]^r\oplus\Z\to \Z[H]^r\oplus \Z\] determined by $f(v_1):=v_1$, $f(v_i):= v_i-v_{i-1}$ for $i=2,\dots,r$, and $f(0,1):=(0,1)$. 
		
		Then $f(\tau)=v_r+(0,-d)=(0,\dots,0,\sigma_{H,d})$, so $f$ induces an isomorphism $(\Z[H]^r\oplus\Z)/\ang{\tau}\cong (\Z[H]^r\oplus\Z)/\ang{f(\tau)}\cong \Z[H]^{r-1}\oplus M_{H,d}$. We conclude that $M_{G,d}\cong \Z[H]^{r-1}\oplus M_{H,d}$, as desired.
	\end{proof}
	
	The following group-theoretic lemma will be used in the proof of \Cref{classify}.
	
	\begin{lemma}\label{restrictive}
		Let $G$ be a finite $p$-group. 
		\begin{enumerate}[label=(\alph*)]
			\item\label{restrictive1} Assume that $G$ does not contain a proper subgroup isomorphic to $\Z/p\Z\times \Z/p\Z$. Then $G$ is cyclic, or $p=2$ and $G=Q_{2^n}$ for some $n\geq 3$.
			\item\label{restrictive2} Assume that every proper subgroup of $G$ is cyclic. Then either $G$ is cyclic or $\Z/p\Z\times \Z/p\Z$, or $p=2$ and $G=Q_8$.
			\item\label{restrictive3} Assume that $p=2$ and that every proper subgroup of $G$ is cyclic or dihedral. Then $G$ is cyclic, dihedral, or $|G|=8$. 
		\end{enumerate}
	\end{lemma}
	
	\begin{proof}
		
		\ref{restrictive1}. Assume that $G$ contains at least two distinct subgroups of order $p$. Since the center $Z(G)$ is not trivial, there exists a subgroup $H$ of $Z(G)$ order $p$. If $H'$ is another subgroup of order $p$, then $\ang{H,H'}\cong \Z/p\Z\times \Z/p\Z$. It follows from the assumption that $G=\Z/p\Z\times\Z/p\Z$. 
		
		Assume now that $G$ contains exactly one subgroup of order $p$. By \cite[Proposition 1.3]{berkovich2008groups}, $G$ is cyclic or $p=2$ and $G$ is a generalized quaternion group: \[Q_{2^{n+1}}:=\ang{a,b\ |\ a^{2^{n-1}}=b^2, a^{2^n}=e, b^{-1}ab=a^{-1}}.\] 
		
		\ref{restrictive2}. If $p$ is odd \ref{restrictive2} immediately follows from \ref{restrictive1}, so assume $p=2$. If $n\geq 4$, by $Q_{2^n}$ contains $Q_8$ as a proper subgroup. It follows from \ref{restrictive1} that if $p=2$ and $G$ is not cyclic, then $G=Q_8$.
		
		\ref{restrictive3}. The conclusion is obvious when $G$ is abelian, so assume that $G$ is not abelian. We may also suppose that $|G|=2^n$ for some $n\geq 4$. If $G$ contains at least one non-abelian proper subgroup, it is dihedral by \cite{miller1907groups}. Assume now that every proper subgroup of $G$ is abelian. By \cite{redei1947schiefe} or \cite[Lemma 4.5]{kang2002trace}, $G$ is isomorphic to $Q_8$, or
		\[A_{u,v}:=\ang{a,b\ |\ a^{2^u}=b^{2^v}=1, ba=a^{2^{u-1}+1}b}\cong \Z/2^u\Z\rtimes \Z/2^v\Z, \] where $u\geq 2$, $v\geq 1$ and $n=u+v$, or 
		\[B_{u,v}:=\ang{a,b,c\ |\ a^{2^u}=b^{2^v}=c^2=1, ba=cab}\cong (\Z/2^u\Z\times \Z/2\Z)\rtimes \Z/2^v\Z\] where $u,v\geq 1$ and $n=u+v+1$. Since $n\geq 4$, both groups contain a subgroup of the form $\Z/4\Z\times \Z/2\Z$, which is neither cyclic nor dihedral, hence $G=Q_8$.
	\end{proof}

	\section{Proof of Theorem \ref{classify}, $p$ odd}
	The next lemma allows us to reduce the problem of the retract rationality of $BT[d]$ to the case when $G$ is a $p$-group.
	
	\begin{lemma}\label{reducetosylow}
		Let $T$ be a $k$-torus split by a group $G$. Assume that $d\mid |G|$ and let $d=p_1^{a_1}\cdots p_r^{a_r}$ be the prime factorization of $d$. For every $i=1,\dots,r$, let $G_{i}$ be a $p_i$-Sylow subgroup of $G$, and set $k_i:=L^{G_i}$. Then $BT[d]$ is retract rational if and only if $BT_{k_i}[p_i^{a_i}]$ is retract rational over $k_i$ for every $i=1,\dots,r$.	
	\end{lemma}
	
	\begin{proof}
		We have $BT[d]\cong BT[p_1^{a_1}]\times\cdots\times BT[p_r^{a_r}]$, so $BT[d]$ is retract rational if and only if $BT[p_i^{a_i}]$ is retract rational for $i=1,\dots,r$, that is, if and only if $T_{p_i^{a_i}}$ is retract rational for every $i$. By \cite[Lemme 9]{colliot1977r}, $T_{p_i^{a_i}}$ is retract rational if and only if $(T_{k_j})_{p_i^{a_i}}= (T_{p_i^{a_i}})_{k_j}$ is retract rational over $k_j$ for every $j=1,\dots,r$. Thus $BT[d]$ is retract rational if and only if $BT_{k_j}[p_i^{a_i}]$ is retract rational over $k_j$ for every $i,j$.
		
		If $i\neq j$, $p_i^{a_i}$ is invertible modulo $|G_j|$, so by \Cref{periodicity}\ref{periodicity4} $BT_{k_j}[p_i^{a_i}]$ is retract rational over $k_j$. It follows that $BT[d]$ is retract rational if and only if  $BT_{k_i}[p_i^{a_i}]$ is retract rational over $k_i$ for every $i$.
	\end{proof}
	Let $L/k$ be a finite Galois extension with Galois group $G$, and let $T:=R^{(1)}_{L/k}(\G_m)$.
	\begin{lemma}\label{psquare} 
		Let $p$ be an odd prime, $G=\Z/p\Z\times \Z/p\Z$. If $p$ is odd, $BT[p]$ is not retract rational.
	\end{lemma}
	
	\begin{proof}
		By \Cref{sha} it suffices to prove that $\Sh^2(\Z/p\Z\times \Z/p\Z,\Z/p\Z)\neq 0$. Let $H_p$ be the Heisenberg group of order $p^3$. The group $H_p$ fits in a central short exact sequence \[1\to \Z/p\Z\to H_p\xrightarrow{\pi} \Z/p\Z\times \Z/p\Z\to 1.\] Let $\alpha\in H^2(\Z/p\Z\times\Z/p\Z,\Z/p\Z)$ be the class corresponding to this extension. Since $p$ is odd, $H_p$ is not commutative, so  $\alpha\neq 0$. If $g\in \Z/p\Z\times\Z/p\Z$, the image of $\alpha$ in $H^2(\ang{g},\Z/p\Z)$ is the class of the extension \[1\to \Z/p\Z \to \pi^{-1}(\ang{g})\to \ang{g}\to 1.\] Since every non-trivial element $g\in H_p$ has order $p$, $\pi^{-1}(\ang{g})\cong \Z/p\Z\times\Z/p\Z$, showing that $\alpha$ restricts to the trivial class in $H^2(\ang{g},\Z/p\Z)$. It follows that $0\neq \alpha\in \Sh^2(\Z/p\Z\times \Z/p\Z,\Z/p\Z)$, as desired.
	\end{proof}
	Note that \Cref{psquare} fails when $p=2$; see \Cref{smallcases}\ref{dih} below. 
	\begin{proof}[Proof of \Cref{classify}, $p$ odd]
		Note that if $p\nmid d$, then $BT[d]$ is stably rational by \Cref{periodicity}\ref{periodicity5}. From now on, we assume that $p\mid d$.
		
		\ref{classify1}. When $G$ is cyclic, every torus split by $G$ is retract rational by \cite[Proposition 2]{colliot1977r}, hence $T_d$ is retract rational. It follows that $BT[d]$ is also retract rational.
		
		\ref{classify2}. If $G$ is not cyclic, by \Cref{restrictive}\ref{restrictive1} it contains a subgroup isomorphic to $\Z/p\Z\times \Z/p\Z$. Since $p\mid d$, the conclusion follows from \Cref{psquare} and \Cref{subgroup}.		
		%If $p^2\mid d$, then by \Cref{periodicity}\ref{periodicity2} $BT[d]$ is stably birational to $T\times BT$. Since $T'$ is rational, by \cite[Proposition 6.1]{scavia2018retract} $BT$ is stably rational, and by \cite[Proposition 2]{colliot1977r} $T$ is not retract rational, hence $BT[d]$ is not retract rational. Finally, if $p\mid d$ but $p^2\nmid d$, then by \Cref{periodicity}\ref{periodicity4} we may assume that $d=p$. In this case, the claim follows from \Cref{psquare} and \Cref{subgroup}. 
	\end{proof}
	
	\section{A flasque resolution of $\hat{T}_d$}
	In order to complete the proof of \ref{classify}, it remains to consider the case $p=2$. In this section we construct an explicit flasque resolution of $\hat{T}_d$. The construction works for any prime $p$ and any $d$, but we will use it only in the case when $p=2$.
	
	We may assume that $G$ is a $p$-group, and by \Cref{periodicity} that $d\mid |G|$. Consider the sequence
	\begin{equation}\label{norm1torsion} 0\to M_d\to \Z[G]\oplus \Z\xrightarrow{\alpha} \Z\to 0,\end{equation}
	where $M_d:=\hat{T}'_d$, obtained by dualizing (\ref{norm1torsion0}). The map $\alpha$ is defined by $\alpha(e_g,0)=1$ for every $g\in G$ and $\alpha(0,1)=-d$, so \[M_d=\set{(a_g;a)\in \Z[G]\oplus\Z: \sum a_g-da=0}.\]	
	We need to construct an explicit coflasque resolution of $M_d$. The standard coflasque resolution (see e.g. \cite[Lemme 3]{colliot1977r}) applied to $M_d$ is too unwieldy for the computations that we want to perform. Thus we produce an ad-hoc coflasque resolution of $M_d$. Consider the short exact sequence	
	\begin{equation}\label{cofl}
	0\to N_d\to P_d\xrightarrow{\beta} M_d\to 0.
	\end{equation}
	Here $P_d:=\Z[G\times G]\oplus R_d$, where $R_d:=\oplus_{G'< G} \Z[G/G']$, the sum being over all subgroups $G'$ of $G$; in particular $P_d$ is a permutation lattice. We denote by $(g,g')$, $g,g'\in G$, the elements of the standard basis of $\Z[G\times G]$. For every subgroup $G'$ of $G$, denote by $e_{G'}$ the element of $P_d$ supported on the summand $\Z[G/G']$ and having a $1$ in the $\Z$-component corresponding to the coset of $G'$ in $G/G'$, and $0$ everywhere else. Recall that $\on{Hom}_G(\Z[G/G'],M_d)=\on{Hom}_{G'}(\Z,M_d)\cong M_d^{G'}$, the isomorphism being given by $\phi\mapsto \phi(e_{G'})\in M_d^{G'}$, so a map $R_d\to M_d$  is specified by the images of all the $e_{G'}$, and for every subgroup $G'$ the image of $e_{G'}$ must be $G'$-invariant. Then $\beta$, as a map $P_d\to \Z[G]\oplus \Z$, is defined by $\beta((g,g'),0):=(e_{g}-e_{g'},0)$ and 
	\begin{align*}
	\beta(e_{G'}):=\begin{cases}
	(\sigma_{G'},|G'|/d) &\text{if $d\mid |G'|$,}\\
	((d/|G'|)\sigma_{G'},1) &\text{if $d\nmid |G'|$.}\\
	\end{cases}
	\end{align*}
	Note that since $G$ is a $p$-group and $d\mid|G|$, either $d\mid|G'|$ or $|G'|\mid d$. To see that $\beta$ maps $P_d$ to $M_d$, observe that clearly $\alpha(\beta((g_1,g_2),0))=0$ and that
	\begin{align*}
	\alpha(\beta(e_{G'}))=\begin{cases}
	|G'|-d\cdot |G'|/d=0 &\text{if $d\mid |G'|$,}\\
	(d/|G'|)\cdot|G'|-d=0 &\text{if $d\nmid |G'|$.}\\
	\end{cases}
	\end{align*}
	Moreover, $\beta(e_{G'})\in M_d^{G'}$, so $\beta$ is $G$-equivariant. If we fix a subgroup $G_0<G$ of order $d$, then $M_d$ is generated as a $G$-module by elements of the form $(e_{g}-e_{g'},0)$, $g,g'\in G$, and $(\sigma_{G_0},1)$. Now, $e_{g}-e_{g'}=\beta((g,g'),0)$, and $(\sigma_{G_0},1)=\beta(e_{G_0})$, so $\beta$ is surjective. We define $N_d:=\on{ker}\beta$. 
	
	\begin{lemma}\label{coflres}
		The $G$-lattice $N_d$ is coflasque, so (\ref{cofl}) is a coflasque resolution of $M_d$. 
	\end{lemma}
	
	\begin{proof}
		Let $G'$ be a subgroup of $G$. Since $P_d$ is a permutation lattice, we have $H^{1}(G',P_d)=0$. Recall that we denote degree $0$ Tate cohomology by $\hat{H}^0$. The cohomology long exact sequence associated to (\ref{cofl}) then yields \[\hat{H}^0(G',P_d)\xrightarrow{\hat{H}^0(\beta)} \hat{H}^0(G',M_d)\to H^{1}(G',N_d)\to 0,\] so to prove that $H^1(G',N_d)=0$ it suffices to show that $\hat{H}^0(\beta)$ is surjective. Since $H^{-1}(G',\Z)=0$, passing to cohomology in (\ref{norm1torsion}) gives \[0\to \hat{H}^{0}(G',M_d)\to\hat{H}^0(G',\Z[G]\oplus \Z)\to\hat{H}^0(G',\Z),\] which can be rewritten as
		\begin{equation}\label{cofl1}0\to \hat{H}^{0}(G',M_d)\to\Z/|G'|\Z\xrightarrow{-d}\Z/|G'|\Z.\end{equation}
		We have $\hat{H}^0(G',P_d)=\hat{H}^0(G',R_d)=\oplus_{G''<G}\hat{H}^0(G',\Z[G/G''])$. To prove that $\hat{H}^0(\beta)$ is surjective it suffices to show that the map \[\gamma:\hat{H}^0(G',\Z[G/G'])\to \hat{H}^0(G',M_d)\] given by the summand relative to $G''=G'$ is surjective. If we view $\hat{H}^0(G',M_d)$ as the $d$-torsion of $\hat{H}^0(G',\Z[G]\oplus \Z)$ via (\ref{cofl1}), then $\gamma(e_{G'})$ coincides with the $\Z$-component of $\beta(e_{G'})\in \Z[G]\oplus \Z$ modulo $|G'|$, that is:
		\begin{align*}
		\gamma(e_{G'})=\begin{cases}
		|G'|/d &\text{if $d\mid |G'|$,}\\
		1 &\text{if $d\nmid |G'|$.}\\
		\end{cases}
		\end{align*}
		In both cases $\gamma(e_{G'})$ generates $\hat{H}^0(G',M_d)$. It follows that $\gamma$ is surjective, as desired. 
	\end{proof}
	
	\begin{prop}\label{cyclic}
		Let $L/k$ be a Galois extension whose Galois group $G$ is a $p$-group, and set $T:=R_{L/k}^{(1)}(\G_m)$. If $BT[d]$ is retract rational for some $d\mid |G|$, then any subgroup of $G$ of order $d$ is cyclic.
	\end{prop}
	
	\begin{proof}
		Assume that $BT[d]$ is retract rational. Then by \cite[Proposition 6.1]{scavia2018retract} $T_d$ is retract rational. By \Cref{coflres}, the dual of (\ref{cofl}) is a flasque resolution of $\hat{T}_d$. By \cite[Theorem 3.14(a)]{saltman1984retract} $N'_d$ is invertible, hence so is $N_d$. This means that there exists a $G$-lattice $U$ such that $N_d\oplus U=P$ for some permutation lattice $P$. In particular, for every subgroup $G'$ of $G$ we have an embedding $H^2(G',N_d)\hookrightarrow H^2(G',P)$. Since ${H}^{-1}(G',P_d)=0$, the long exact sequence associated to (\ref{cofl}) shows that $H^1(G',M_d)$ embeds in $H^2(G',N_d)$, hence in $H^2(G',P)$. The long exact sequence associated to (\ref{norm1torsion}) gives \[H^0(G',\Z[G]\oplus \Z)\to H^0(G',\Z)\to H^1(G',M_d)\to 0.\] This shows that if $d\mid |G'|$, then $H^1(G',M_d)\cong \Z/d\Z$. If $|G'|=d$, it follows that $H^2(G',P)$ contains an element of order $|G'|$. Now, $P$ is a permutation $G'$-module, that is, a direct sum of $G'$-modules of the form $\Z[G'/G'']$, where $G''$ is a subgroup of $G'$. Since \[H^2(G',\Z[G'/G''])=H^2(G'',\Z)=H^1(G'',\Q/\Z)\cong G''\] for every subgroup $G''$ of $G'$, the existence of an element of order $|G'|$ in $H^2(G',P)$ implies that $G'$ must be cyclic.	
	\end{proof}
	Note that, when $p$ is odd, \Cref{cyclic} is weaker than \Cref{classify}\ref{classify2}, which we have already proved. We will use \Cref{cyclic} to prove \Cref{classify} when $p=2$.
	
	\section{Proof of Theorem \ref{classify}, $p=2$}
	For every even integer $r\geq 4$, we denote by $D_{r}$ the dihedral group of $r$ elements. In particular $D_4\cong \Z/2\Z\times \Z/2\Z$.
	
	\begin{lemma}\label{smallcases}
		Let $L/k$ be a Galois extension with Galois group $G$, and let $T:=R^{(1)}_{L/k}(\G_m)$.
		\begin{enumerate}[label=(\alph*)]
			\item\label{dih} Assume that $G=D_{2^n}$ for some $n\geq 1$. Then $BT[2]$ is retract rational. 
			\item\label{quat1} Let $G=Q_8$ be the quaternion group. Then $BT[2]$ is not retract rational.
			\item\label{c4c2} Let $G=\Z/4\Z\times \Z/2\Z$. Then $BT[2]$ is not retract rational.
		\end{enumerate}
	\end{lemma}
	
	\begin{proof}
		\ref{dih} We start by showing that $\Sh^2(D_{2^n},\Z/2\Z)=0$. Recall that the non-trivial element $\alpha$ of $H^2(D_{2^n},\Z/2\Z)\cong \Z/2\Z$ is given by the central extension \[1\to \Z/2\Z\to Q_{2^{n+1}}\to D_{2^n}\to 1.\] Let $g\in Q_{2^{n+1}}$ be an element of order $2^n$, and let $\cl{g}\in D_{2^n}$ be the image of $g$. Then $\cl{g}$ has order $2^{n-1}$ and the following sequence is exact:
		\[1\to \Z/2\Z\to \ang{g}\to \ang{\cl{g}}\to 1.\] This shows that the restriction map $H^2(D_{2^n},\Z/2\Z)\to H^2(\ang{\cl{g}},\Z/2\Z)$ sends $\alpha$ to a non-zero class. This means that $\Sh^2(D_{2^n},\Z/2\Z)=0$, as claimed. 
		
		The conclusion now follows from known results. If $G'$ is a dihedral subgroup of $D_{2^n}$, then we have just shown that $\Sh^2(G',\Z/2\Z)=0$. By (\ref{shah1}), we obtain $H^1(G', [\hat{T}_2]^{fl})=0$. If $G'$ is a cyclic subgroup of $G$, then $H^1(G',[\hat{T}_2]^{fl})=0$ by \cite[Proposition 2]{colliot1977r}. Since every subgroup of $G$ is either cyclic or dihedral, it follows that $[\hat{T}_2]^{fl}$ is coflasque. By \cite[\S 4, Theorem 5]{voskresenskii2011algebraic}, $[\hat{T}_2]^{fl}$ is invertible, so by \cite[Proposition 3.14(a)]{saltman1984retract} $T_2$ is retract rational. Since $R_{L/k}(\G_m)\to T_2$ is a generic $T[2]$-torsor, $BT[2]$ is stably rational. 
		
		\ref{quat1}. By \Cref{sha} it suffices to show that $\Sh^2(Q_8,\Z/2\Z)\neq 0$. Recall that $H^2(Q_8,\Z/2\Z)\cong \Z/2\Z\times \Z/2\Z$, and the three non-trivial classes correspond to central extensions of the form \begin{equation}\label{quaternion1}1\to \Z/2\Z\xrightarrow{\iota} \Gamma\xrightarrow{\pi} Q_8\to 1,\end{equation} where \[\Gamma:=\ang{a,b\ |\ a^4=b^4=e, bab^{-1}=a^{-1}}\cong \Z/4\Z\rtimes \Z/4\Z\] is the unique non-trivial semidirect product of $\Z/4\Z$ and $\Z/4\Z$. Note that $Z(\Gamma)=\set{e,a^2,b^2,a^2b^2}\cong \Z/2\Z\times \Z/2\Z$, and $\iota(1)=a^2b^2$. Let $\alpha\in H^2(Q_8,\Z/2\Z)$ be the class of (\ref{quaternion1}), in the case when $\pi(a)=i$ and $\pi(b)=j$. Since \[\pi^{-1}(\ang{-1})=\ang{a^2,b^2}\cong \Z/2\Z\times \Z/2\Z,\] the sequence \[1\to \Z/2\Z\to \pi^{-1}(\ang{-1})\to \ang{-1}\to 1\] splits, so $\alpha$ restricts to $0$ in $H^1(\ang{-1},\Z/2\Z)$. Moreover, the subgroups\[\pi^{-1}(\ang{i})=\ang{a,b^2},\quad \pi^{-1}(\ang{j})=\ang{a,b^2},\quad \pi^{-1}(\ang{i})=\ang{ab,a^2b^2}\] all admit a set of two commuting generators of orders $4$ and $2$, hence they are isomorphic to $\Z/4\Z\times \Z/2\Z$. It follows that $\alpha$ restricts to $0$ on $H^2(G',\Z/2\Z)$ for every proper subgroup $G'$ of $Q_8$, so $0\neq \alpha\in \Sh^2(Q_8,\Z/2\Z)$.
		
		\ref{c4c2}. Let $a:=(1,0)$ and $b:=(0,1)$ be two generators of $G=\Z/4\Z\times \Z/2\Z$. Let $\Gamma:=G\rtimes \Z/2\Z$, where the generator of $\Z/2\Z$ acts on $G$ by sending $a\mapsto ab$ and $b\mapsto b$. The center of $\Gamma$ is $\ang{(2a,0),(b,0)}\cong \Z/2\Z\times \Z/2\Z$, and we have a short exact sequence \begin{equation}\label{z4z2}1\to \Z/2\Z\xrightarrow{\iota} \Gamma\xrightarrow{\pi}G\to 1,\end{equation} where $\iota(1):=b$. Let $\gamma\in \Z/4\Z\times \Z/2\Z$, and consider the exact sequence
		\[1\to \Z/2\Z\to \Gamma'\to\ang{\gamma}\to 1,\] where $\Gamma':=\pi^{-1}(\ang{\gamma})$. It is easy to check that $(b,0)$ is not the square of any element of $\Gamma$, hence $\pi^{-1}(\ang{\gamma})$ is a non-cyclic extension of $\ang{\gamma}$ by $\Z/2\Z$. If $\ang{\gamma}\cong\Z/2\Z$, then necessarily $\Gamma'\cong \Z/2\Z\times \Z/2\Z$, and if $\ang{\gamma}\cong \Z/4\Z$, then $\Gamma'\cong \Z/4\Z\times \Z/2\Z$. This shows that the class of (\ref{z4z2}) in $H^2(\Z/4\Z\times \Z/2\Z,\Z/2\Z)$ is a non-trivial element of $\Sh^2(\Z/4\Z\times \Z/2\Z,\Z/2\Z)$.
	\end{proof}
	
	In the next proposition we collect the two pieces of the information needed to complete the proof of \Cref{classify} that we obtained via a computer calculation. We use Algorithm F2 of \cite[\S 5.2]{hoshi2017rationality}, implemented in the computer algebra system GAP. I thank Thomas R\"ud for helping me understand the code in \cite{hoshi2017rationality}.
	\begin{prop}\label{gap}
		If $G=Q_8$, then $BT[4]$ is not retract rational.
		
		If $G=(\Z/2\Z)^3$, then $BT[2]$ is not retract rational.	
	\end{prop}
	
	\begin{proof}
		If $(G,d)$ equals $(Q_8,4)$ or $((\Z/2\Z)^3,2)$, it is enough to check that $T_d$ is not retract rational, that is, that $[T_d]^{fl}$ is not invertible. Using the presentation (\ref{norm1torsion}), this can be done using Algorithm F2 of \cite[\S 5.2]{hoshi2017rationality}.
	\end{proof}
	
	\begin{rmk}
		Similarly to the proof of \Cref{smallcases}\ref{dih} (or indeed by another GAP computation), one may show that $\Sh^2(Q_8,\Z/4\Z)=0$ and $\Sh^2((\Z/2\Z)^3,\Z/2\Z)=0$. In other words, in the two cases considered in \Cref{gap}, $T_d$ is coflasque but not retract rational; cf. \cite[\S 4 Example 2, Theorem 5]{voskresenskii2011algebraic}. We will not make use of this in the sequel.
	\end{rmk}
	
	\begin{proof}[Proof of \Cref{classify}, $p=2$]
		If $d$ is odd, the conclusion follows from \Cref{periodicity} \ref{periodicity5} and \ref{periodicity4}. From now on, we assume that $d$ is even.
		
		\ref{classify3}. The proof is the same as that of \ref{classify1}. When $G$ is cyclic, every torus split by $G$ is retract rational by \cite[Proposition 2]{colliot1977r}, hence $T_d$ is retract rational, and so $BT[d]$ is also retract rational.
		
		\ref{classify4}. When $4\nmid d$, $BT[d]$ is retract rational by \Cref{smallcases}\ref{dih}, so assume $4\mid d$. If $G=D_4=\Z/2\Z\times \Z/2\Z$, then by \Cref{periodicity}\ref{periodicity2} $BT[d]$ is stably birational to $T$, since $T'$ is rational. The torus $T$ is not retract rational by \cite[Proposition 2]{colliot1977r}, hence $BT[d]$ is not retract rational in this case. If $G=D_{2^n}$ for some $n\geq 3$, then $G$ contains a subgroup isomorphic to $\Z/2\Z\times \Z/2\Z$, so the claim follows from \Cref{subgroup}.
		
		\ref{classify5}. Assume that $BT[d]$ is retract rational for some $d$ divisible by $4$. We prove by induction on $n$ that $G$ is either cyclic or dihedral, where $|G|=2^n$. If $n=1$ there is nothing to prove. If $n=2$, then by assumption $|G|=4$ divides $d$, hence by \Cref{periodicity}\ref{periodicity2} $BT[d]$ is stably birational to $T$. If $G=\Z/2\Z\times \Z/2\Z$, $T$ is not retract rational by \cite[Proposition 2]{colliot1977r}, hence $G=\Z/4\Z$. When $n=3$, by the case $n=2$ and \Cref{subgroup}, $G$ does not contain a subgroup isomorphic to $\Z/2\Z\times \Z/2\Z$, hence $G=\Z/8\Z$ or $Q_8$. When $G=Q_8$, $BT[d]$ is not retract rational by \Cref{gap}. If $n\geq 4$, by \Cref{subgroup} and the inductive assumption every subgroup of $G$ is either cyclic or dihedral, hence by \Cref{restrictive}\ref{restrictive3} $G$ is also cyclic or dihedral. 
		
		%By the inductive assumption, every proper subgroup of $G$ is either cyclic or dihedral.  \Cref{gap} and \Cref{subgroup}, $G'$ is not isomorphic to $Q_8$. By \Cref{cyclic}, every subgroup of $G$ of order $4$ is cyclic, hence $G'$ can only be cyclic or $D_8$.  By \cite[Corollary 4.7]{craven2008theory} $G$ is cyclic or generalized quaternion. We excluded the first case, and in the second case $G=Q_{2^n}$ contains a subgroup isomorphic to $Q_8$. By \Cref{gap}, $BT[4]$ is not retract rational when $G=Q_8$, hence by \Cref{subgroup} $BT[4]$ is not retract rational when $G=Q_{2^n}$ as well. 
		
		Assume now that $BT[d]$ is retract rational for some $d\equiv 2\pmod 4$. We prove that $G$ is either cyclic or dihedral by induction on $n$, where $|G|=2^n$. If $n=1,2$ there is nothing to prove. If $n=3$, then by \Cref{smallcases} $G'$ is not isomorphic to $Q_8$ or $\Z/4\Z\times \Z/2\Z$, and by \Cref{gap} it is not isomorphic to $(\Z/2\Z)^3$, hence it is either $\Z/8\Z$ or $D_8$, as desired. If $n\geq 4$, by inductive assumption every proper subgroup of $G$ is either cyclic or dihedral. By \Cref{restrictive}\ref{restrictive3}, $G$ is either cyclic or dihedral. 
	\end{proof}
	
	\begin{example}
		Let $L/k$ be a Galois extension with Galois group a symmetric group $S_n$, $d\geq 1$, and $p$ a prime. Recall that if $n/2< p\leq n$, so that $p\mid n!$ but $p^2\nmid n!$, then any $p$-Sylow subgroup of $S_n$ is cyclic of order $p$, and that if $p<n/2$ then any $p$-Sylow subgroup of $S_n$ contains a subgroup of the form $\Z/p\Z\times \Z/p\Z$. The $2$-Sylow subgroups of $S_n$ are cyclic for $n=2,3$, dihedral for $n=4,5$, and neither cyclic nor dihedral for $n\geq 6$ (for example because they contain subgroups of the form $D_8\times \Z/2\Z$). Using \Cref{norm1torsion} and \Cref{reducetosylow}, we obtain:
		\begin{itemize}
			\item If there exists an odd prime $p$ such that $2p\leq n$ and $p\mid d$, then $BT[d]$ is not retract rational;
			\item If $n=4,5$, and $4\mid d$, then $BT[d]$ is not retract rational;
			\item If $n\geq 6$, and $d$ is even, then $BT[d]$ is not retract rational. 
		\end{itemize}
		In all other cases, $BT[d]$ is retract rational. 
		
		Assume further that $d\mid |S_n|=n!$; by \Cref{periodicity} one can reduce to this case by multiplying by some invertible element in $\Z/e(T)\Z$. Define $\sigma:\mathbb{N}\to\mathbb{N}$ by \[\sigma(1)=1,\ \sigma(2)=2,\ \sigma(3)=6,\ \sigma(4)=6,\ \sigma(5)=30\] and \[\sigma(n)=\prod_{n/2<p\leq n}p,\qquad n\geq 6.\] Then we have just shown that $BT[d]$ is $p$-retract rational if and only if $d\mid \sigma(n)$. 
	\end{example}
	
	\section{Application to the Grothendieck ring of stacks}\label{grotring}
	%The purpose of this section is to combine \Cref{periodicity} and the example of \cite[Theorem 1.5]{scavia2018motivic} to give a more enlightening proof of \cite[Theorem 1.6]{scavia2018motivic}.
	
	The ring $K_0(\on{Stacks}_k)$ was introduced by Ekedahl in \cite{ekedahl2009grothendieck}. By definition, it is the quotient of the free abelian group on the equivalence classes $\set{\mc{X}}$ of all algebraic stacks $\mc{X}$ of finite type over $k$ with affine stabilizers, by the scissor relations $\set{\mc{X}}=\set{\mc{Y}}+\set{\mc{X}\setminus \mc{Y}}$ for every closed substack $\mc{Y}\c \mc{X}$ and the relations $\set{\mc{E}}=\set{\mc{X}\times \A^n_k}$ for every vector bundle $\mc{E}$ of rank $n$ over $\mc{X}$. The product is defined by $\set{\mc{X}}\cdot\set{\mc{Y}}:=\set{\mc{X}\times_k\mc{Y}}$, and extended by linearity. 
	
	If $G$ is a finite or connected group, there appears to be connection between properties of $\set{BG}$ and the Noether Problem for $G$; as of now, the link between the two is largely unexplained in either direction. Consider the following equations in $K_0(\on{Stacks}_k)$, called \emph{expected class formulas}: 
	\begin{align*}
	&\set{BG}=1, &\text{$G$ finite group},\\
	&\set{BG}=\set{G}^{-1}, &\text{$G$ connected linear group}. 
	\end{align*}
	As Ekedahl shows in \cite{ekedahl2009geometric}, it frequently happens that $\set{BG}=1$ for a finite group $G$. This is true, for example, if $G$ is a symmetric group; see \cite[Theorem 4.3]{ekedahl2009geometric}. It is striking that all the known counterexamples $G$ to $\set{BG}=1$ are also counterexamples to the Noether Problem.
	
	If $G$ is a connected group and $k$ is algebraically closed, no counterexample to $\set{BG}\set{G}=1$ is known; this is again in line with the Noether Problem, for which no negative answer is known among connected groups. In \cite[Theorem 1.5]{scavia2018motivic}, we exhibited the first connected counterexample $T$ to the expected class formula, in the case when $k$ is finitely generated over $\Q$. More precisely, $T:=R^{(1)}_{E_1\times E_2/k}(\G_m)$, where $E_1$ and $E_2$ are distinct quadratic extensions of $k$; see \cite[\S 3]{scavia2018motivic}. Using this result, we showed that $\set{BT'[2d]}=\set{BT}^{-1}\set{T}^{-1}\neq 1$ in \cite[Theorem 1.6]{scavia2018motivic} for every $d\geq 1$. As usual, here $T'=R_{E/k}(\G_m)/\G_m$ is the dual torus of $T$. On the other hand, one has $\set{BT[2d-1]}=1$ for every $d\geq 1$. \Cref{periodicity} provides a conceptual explanation for this periodicity. Moreover, it allow us to compute $\set{BT[n]}$ for every $n\geq 1$.
	
	\begin{prop}\label{grotper}
		Let $T$ be a $k$-torus, and $m,n\geq 1$ be integers.
		\begin{enumerate}[label=(\alph*)]
			\item\label{grotper1} If $n\equiv \pm m \pmod{e(T)}$, then $\set{BT[n]}=\set{BT[m]}$.
			\item\label{grotper2} If $n\equiv \pm 1 \pmod{e(T)}$, then $\set{BT[n]}=1$.
			\item\label{grotper3} If $e(T)\mid n$, then $\set{BT[n]}=\set{BT}\set{T}$.
		\end{enumerate}
	\end{prop}
	
	%If $T=R^{(1)}_{E_1\times E_2/k}(\G_m)$, then $e(T)=2$, so the class $\set{BT[n]}$ is $2$-periodic in $n$; see \Cref{grotring}.
	
	\begin{proof}
		Fix $d\geq 1$, and consider a diagram (\ref{noethertorsion1}) for $T$. Since the quasi-split torus $S$ is special, by \cite[Corollary 2.4]{bergh2015motivic} we have $\set{BS}\set{S}=1$ in $K_0(\on{Stacks}_k)$, hence $\set{S}$ is invertible. By \cite[Proposition 2.9]{bergh2015motivic} we have \begin{equation}\label{grotpereq1}\set{BT}=\set{Q}/\set{S}\end{equation} and \begin{equation}\label{grotpereq2}\set{BT[d]}=\set{T_d}/\set{S}\end{equation} in $K_0(\on{Stacks}_k)$.
		
		\ref{grotper1}. If $n\equiv m\pmod{e(T)}$, by \Cref{multiple} $T_n\cong T_m$, hence the claim follows from (\ref{grotpereq2}).	
		
		\ref{grotper2}. If $n\equiv \pm 1 \pmod{e(T)}$ then by \ref{grotper1} $\set{BT[n]}=\set{BT[1]}=1$.
		
		\ref{grotper3}. If $e(T)\mid n$, by \Cref{multiple} $T_n\cong T\times Q$, so $\set{T_d}=\set{T}\set{Q}$. Combining this with (\ref{grotpereq1}) and (\ref{grotpereq2}) (for $n=d$), we obtain \[\set{BT[n]}=\set{T_n}/\set{S}=\set{T}\set{Q}/\set{S}=\set{T}\set{BT}.\qedhere\]
	\end{proof}
	
	\begin{prop}
		Let $T:=R^{(1)}_{E/k}(\G_m)$, where $E$ is a product of two separable quadratic extensions of $k$. For every $d\geq 1$, we have \[\set{BT[2d-1]}=1,\qquad \set{BT[2d]}=\set{BT}\set{T}\] in $K_0(\on{Stacks}_k)$.	
	\end{prop}
	
	\begin{proof}
		Using the norm exact sequence \[0\to\Z\to \Z[C_2]\oplus \Z[C_2]\to \hat{T}\to 0,\] the argument of \cite[Example 5.1]{merkur2010periods} adapts verbatim to a proof that $e(T)=2$. The conclusion now follows from \Cref{grotper}.
	\end{proof}
	
	%The previous proposition gives a more conceptual proof of \cite[Proposition 5.1]{scavia2018motivic} and, as a consequence, of \cite[Theorem 1.6]{scavia2018motivic}.
	
	Since $\set{BT}\set{T}\neq 1$, the previous proposition gives a more conceptual proof of \cite[Theorem 1.6]{scavia2018motivic}.

	\section*{Acknowledgments}
	I would like to thank my advisor Zinovy Reichstein for useful comments on this work, and Thomas R\"ud for helping me with GAP.
	\bibliography{bibliografia}
	\bibliographystyle{plain}
	
\end{document}